\documentclass[11pt,a4paper]{amsart}
\usepackage{graphicx,multirow,array,amsmath,amssymb}

\newtheorem{theorem}{Theorem}

\newtheorem{lemma}[theorem]{Lemma}

\begin{document}

\title{Bounds on multiple self-avoiding polygons}

\author[K. Hong]{Kyungpyo Hong}
\address{National Institute for Mathematical Sciences, Daejeon 34047, Korea}
\email{kphong@nims.re.kr}
\author[S. Oh]{Seungsang Oh}
\address{Department of Mathematics, Korea University, Seoul 02841, Korea}
\email{seungsang@korea.ac.kr}

\thanks{The corresponding author(Seungsang Oh) was supported by the National Research Foundation of Korea(NRF) grant funded by the Korea government(MSIP) (No. NRF-2017R1A2B2007216).}

\maketitle

\begin{abstract}
A self-avoiding polygon is a lattice polygon consisting of a closed self-avoiding walk on a square lattice.
Surprisingly little is known rigorously about the enumeration of self-avoiding polygons,
although there are numerous conjectures that are believed to be true
and strongly supported by numerical simulations.
As an analogous problem of this study, 
we consider multiple self-avoiding polygons in a confined region, 
as a model for multiple ring polymers in physics.
We find rigorous lower and upper bounds of the number $p_{m \times n}$ 
of distinct multiple self-avoiding polygons in the $m \times n$ rectangular grid on the square lattice.
For $m=2$, $p_{2 \times n} = 2^{n-1}-1$.
And, for integers $m,n \geq 3$,
$$2^{m+n-3} \left(\frac{17}{10}\right)^{(m-2)(n-2)} \ \leq \ p_{m \times n} \ \leq \
2^{m+n-3} \left(\frac{31}{16}\right)^{(m-2)(n-2)}.$$
\end{abstract}

\section{Introduction} \label{sec:intro}

The enumeration of self-avoiding walks and polygons is
one of the most important and classic combinatorial problems~\cite{G, MS}.
These were first introduced by the chemist Paul Flory~\cite{F} as models of polymers in dilute solution.
Determining the exact number of self-avoiding walks and polygons is still unsolved,
although there are mathematically proved methods for approximating them.

A particularly interesting polygon model of a ring polymer with excluded volume is a lattice polygon
which places in a regular lattice, usually the two dimensional square lattice or the three dimensional cubic lattice.
Here we consider the problem of self-avoiding polygons (SAP) on the square lattice $\mathbb{Z}^2$.
Let $p_n$ denote the number of distinct SAPs of length $n$
counted up to translational invariance on the square lattice $\mathbb{Z}^2$.
Hammersley~\cite{H} proved that the number $p_n$ grows exponentially:
more precisely the limit $\mu = \lim_{n \rightarrow \infty} p_{2n}^{\frac{1}{2n}}$ is known to exist.
Furthermore it is generally believed~\cite{MS} that $p_{2n} \sim \mu^{2n} n^{\alpha -3}$
as $n \rightarrow \infty$.
Here $\mu$ is called the {\em connective constant\/} of the lattice, and $\alpha$ is the {\em critical exponent\/}.
The reader can find more details in~\cite{J}.

In this paper, we are interested in another point of view of scaling arguments of multiple polygons
on the square lattice, related to the size of a rectangle containing them instead of their length;
see Figure~\ref{fig1}.
Let $\mathbb{Z}_{m \times n}$ denote the $m \times n$ rectangular grid on $\mathbb{Z}^2$,
and let $p_{m \times n}$ be the number of distinct multiple self-avoiding polygons (MSAP) 
in $\mathbb{Z}_{m \times n}$.
Here two MSAPs are considered to be different even though one can be translated upon the other.
Note that in physics they serve as a model for multiple ring polymers in a confined region.

\begin{figure}[h]
\includegraphics{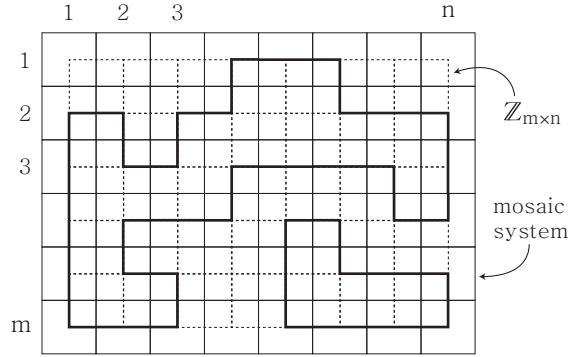}
\caption{Two different viewpoints of a MSAP model
in the confined square lattice $\mathbb{Z}_{m \times n}$
and in the mosaic system (explained in Section 2).}
\label{fig1}
\end{figure}

It is relatively easy to calculate that $p_{2 \times n} = 2^{n-1}-1$ for $m=2$.
But, for larger $m,n$ of $p_{m \times n}$,
the problem becomes increasingly difficult due to its non-Markovian nature.
The main purpose of this paper is to establish rigorous lower and upper bounds for $p_{m \times n}$.

\begin{theorem} \label{thm:polygon}
For integers $m,n \geq 3$,
$$2^{m+n-3} \left(\frac{17}{10}\right)^{(m-2)(n-2)} \ \leq \ p_{m \times n} \ \leq \
2^{m+n-3} \left(\frac{31}{16}\right)^{(m-2)(n-2)}.$$
\end{theorem}

Note that various types of single self-avoiding walks in a confined square lattice were investigated in~\cite{BGJ},
particularly a class of self-avoiding walks that start at the origin $(0,0)$, end at $(n,n)$,
and are entirely contained in the square $[0,n] \times [0,n]$ on $\mathbb{Z}^2$.
The number of distinct walks is known to grow as $\lambda^{n^2 + o(n^2)}$.
They estimate $\lambda = 1.744550 \pm 0.000005$
as well as obtain strict upper and lower bounds, $1.628 < \lambda < 1.782$.
In our model, 
$$1.7 \leq \lim_{n \rightarrow \infty} (p_{n \times n})^{1/n^2} \leq 1.9375,$$ 
provided the limit exists.

\section{Adjusting to the mosaic system}

A mosaic system is introduced by Lomonaco and Kauffman~\cite{LK} to give
a precise and workable definition of quantum knots.
This definition is intended to represent an actual physical quantum system.
The definition of quantum knots was based on the planar projections of knots and the Reidemeister moves.
They model the topological information in a knot by a state vector in a Hilbert space
that is directly constructed from knot mosaics.
Recently Hong, Lee, Lee and Oh announced several results on the enumeration
of various types of knot mosaics in the confined mosaic system in the series of papers
\cite{HLLO1, HLLO2, LHLO, OHLL}.

We begin by explaining the basic notion of mosaics modified for polygons in $\mathbb{Z}_{m \times n}$.
The following seven symbols are called {\em mosaic tiles\/} (for polygons).
In the original definition in mosaic theory, there are eleven types of mosaic tiles
allowing four more mosaic tiles with two arcs.

\begin{figure}[h]
\includegraphics{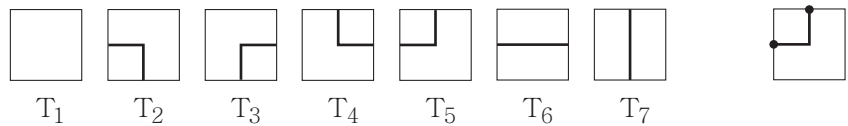}
\caption{Seven mosaic tiles modified for polygons and connection points in a mosaic tile.}
\label{fig2}
\end{figure}

For positive integers $m$ and $n$,
an {\em $(m,n)$-mosaic\/} is an $m \! \times \! n$ matrix $M=(M_{ij})$ of mosaic tiles.
The trivial mosaic is a mosaic whose entries are all $T_1$.
A {\em connection point\/} of a mosaic tile is defined as the midpoint of a tile edge
that is also the endpoint of a portion of graph drawn on the tile as shown in the rightmost tile in Figure~\ref{fig2}.
Note that $T_1$ has no connection point and each of the six mosaic tiles $T_2$ through $T_7$ have two.
A mosaic is called {\em suitably connected\/} if any pair of mosaic tiles
lying immediately next to each other in either the same row or the same column
have or do not have connection points simultaneously on their common edge.
A {\em polygon $(m,n)$-mosaic\/} is a suitably connected $(m,n)$-mosaic
that has no connection point on the boundary edges.
Examples in Figure~\ref{fig3} are a non-polygon $(4,4)$-mosaic and a polygon $(4,4)$-mosaic.

\begin{figure}[h]
\includegraphics{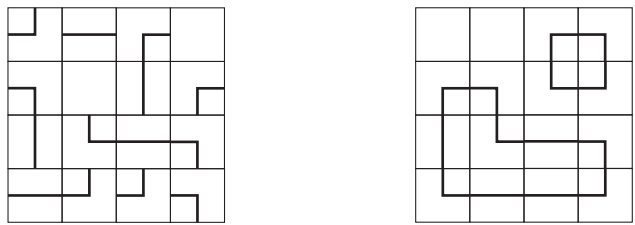}
\caption{Examples of a non-polygon $(4,4)$-mosaic and a polygon $(4,4)$-mosaic.}
\label{fig3}
\end{figure}

As drawn by solid line segments in Figure~\ref{fig1},
we can consider a MSAP as a polygon $(m,n)$-mosaic by shifting
the rectangular grid $\mathbb{Z}_{(m+1) \times (n+1)}$ horizontally and vertically by $-\frac{1}{2}$.
In the mosaic system, polygons transpass unit length edges of the mosaic system
and run through the centers of unit squares.
The following one-to-one conversion arises naturally. \\

\noindent {\bf One-to-one conversion\/}
There is a one-to-one correspondence between  
MSAPs in $\mathbb{Z}_{m \times n}$ and polygon $(m,n)$-mosaics, except for the trivial mosaic. \\

Note that the trivial mosaic contains no graph, so is not counted in $p_{m \times n}$.

\section{Quasimosaics and growth ratios}

In this section, we define a modified version of quasimosaics, which were introduced in~\cite{HLLO2},
and their growth ratios.
We arrange all mosaic tiles as a sequence such that
their pair-indices of tiles are ordered as $(1,1)$, $(1,2)$, $(2,1)$, $(1,3)$, $(2,2)$, $(3,1)$, etc.,
and finished at $(m,n)$.
More precisely, the pair-index $(i,j)$ follows $(i-1,j+1)$ if $i > 1$ and $j < n$,
or otherwise, either $(i+j-2,1)$ for $i+j-2 \leq m$ or $(m,i+j-m-1)$ for $i+j-2 > m$.
Let $a(i,j)$ denote the predecessor of the pair-index $(i,j)$ in the sequence.

An $(i,j)$-{\em quasimosaic\/} is a portion of a polygon $(m,n)$-mosaic obtained by taking
all mosaic tiles $M_{1,1}$ through $M_{i,j}$ in the sequence as drawn in Figure~\ref{fig4}.
Note that a quasimosaic is also suitably connected.
Its $(i,j)$-entry $M_{i,j}$ is called the {\em leading mosaic tile\/} of the $(i,j)$-quasimosaic.
Furthermore we define two kinds of cling mosaics of the $(i,j)$-quasimosaic.
An {\em $l$-cling mosaic\/} for $M_{i,j}$ is a submosaic consisting of three or fewer mosaic tiles
$M_{i,j-2}$, $M_{i,j-1}$ and $M_{i+1,j-2}$ (they may not exist when $j = 1$ or 2).
And a {\em $t$-cling mosaic\/} is a submosaic consisting of five or fewer mosaic tiles
$M_{i-2,j}$, $M_{i-2,j+1}$, $M_{i-2,j+2}$, $M_{i-1,j}$ and $M_{i-1,j+1}$.
The letters $l$- and $t$- mean the left and the top, respectively.
The leftmost and the top boundary edges of cling mosaics
that are not contained in the boundary edges of the mosaic system are called {\em contact edges.\/}

\begin{figure}[h]
\includegraphics{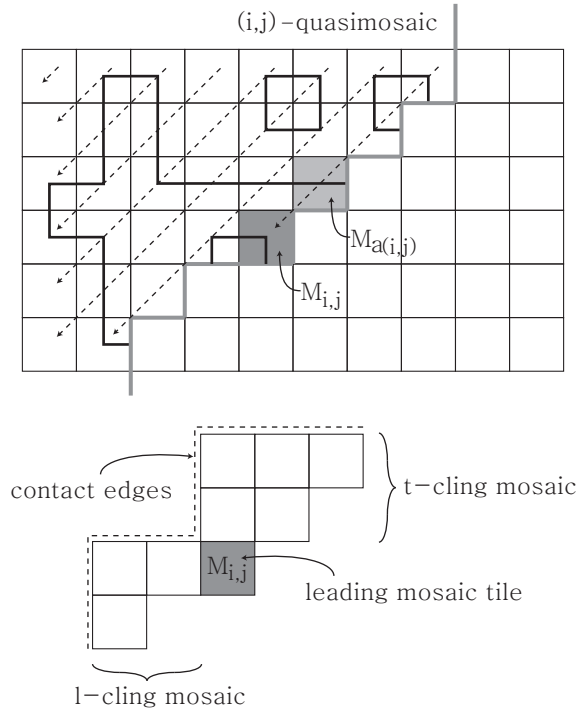}
\caption{A $(4,5)$-quasimosaic and two cling mosaics.}
\label{fig4}
\end{figure}

Let $Q_{i,j}$ denote the set of all possible $(i,j)$-quasimosaics.
By definition, $Q_{m,n}$ is the set of all polygon $(m,n)$-mosaics.
It is an exercise for the reader to show that 
$|Q_{1,1}|=2$, $|Q_{1,2}|=4$, $|Q_{2,1}|=8$, $|Q_{1,3}|=16$, $|Q_{2,2}|=28$ 
and $|Q_{3,1}|=56$, provided that $m,n \geq 4$.
We will construct $Q_{m,n}$ from $Q_{1,1}$ by adding leading mosaic tiles inductively.
Focus on the ratios of growth of the number of sets at each step.
Define a {\em growth ratio\/} $r_{i,j}$ of the set $Q_{i,j}$ over $Q_{a(i,j)}$ as 
$$r_{i,j} = \frac{|Q_{i,j}|}{|Q_{a(i,j)}|},$$
with the assumption that $|Q_{a(1,1)}|=1$.
Thus $r_{1,1}=2$, $r_{1,2}=2$, $r_{2,1}=2$, $r_{1,3}=2$, $r_{2,2}=\frac{7}{4}$, and $r_{3,1}=2$.
By definition,
\begin{equation} \label{eq:1}
p_{m \times n} = |Q_{m,n}| -1 = \prod_{i,j} r_{i,j} -1.
\end{equation}

For simplicity of exposition, a mosaic tile is called {\em $l$-cp\/} if it has a connection point
on its left edge,
and, similarly, {\em $t$, $r$,\/} or {\em $b$-cp\/} when on its top, right, or bottom edge, respectively.
Sometimes we use two letters, for example, $lt$-cp in the case of both $l$-cp and $t$-cp.
Also, we use the sign $\hat{}$ \/ for negation so that, for example,
$\hat{t}$-cp means not $t$-cp,
$\hat{l} \hat{t}$-cp means both $\hat{l}$-cp and $\hat{t}$-cp, and
$\widehat{lt}$-cp (which is differ from $\hat{l} \hat{t}$-cp) means not $lt$-cp,
i.e., $\hat{l} t$, $l \hat{t}$, or $\hat{l} \hat{t}$-cp.

\begin{lemma} \label{lem:choice}
For positive integers $i,j$, $M_{ij}$ is either $T_1$ or $T_3$ if it is $\hat{l} \hat{t}$-cp,
either $T_2$ or $T_6$ if $l \hat{t}$-cp,
either $T_4$ or $T_7$ if $\hat{l} t$-cp, and
$T_5$ if $lt$-cp.
Therefore, each $M_{ij}$ has two choices of mosaic tiles if it is $\widehat{lt}$-cp,
and the unique choice if it is $lt$-cp.
\end{lemma}

Remark that we easily find rough bounds of $r_{i,j}$.
Each $a(i,j)$-quasimosaic in $Q_{a(i,j)}$ can be extended to
either one or two $(i,j)$-quasimosaics in $Q_{i,j}$ by choosing
the leading mosaic tile $M_{i,j}$ being suitably connected according to Lemma~\ref{lem:choice}.
Thus, $|Q_{a(i,j)}| \leq |Q_{i,j}| \leq 2 |Q_{a(i,j)}|$,
and so we have rough bounds of the growth ratio:
$$1 \leq r_{i,j} \leq 2.$$

\section{Investment of cling mosaics and cp-ratios}

We can mark at a mosaic tile edge on a cling mosaic with an `x'
if it does not have a connection point and with an `o' if it has.
Sometimes we use a sequence of x's and o's to mark several edges together, like $e_1 e_2 =$ xo,
which means that the edge $e_1$ does not have a connection point but the edge $e_2$ does.

Now we classify all $l$-cling mosaics into five types $U_1 \sim U_5$,
and all $t$-cling mosaics into eight types $V_1 \sim V_8$ as drawn in Figure~\ref{fig5}.
In each type, the bold edges $e_l$ and $e_t$ indicate the left and the top edges
of the leading mosaic tile, respectively;
the edges $e_i$'s indicate the contact edges,
and the edges marked by x lie in the boundary of the mosaic system
(so these have no connection point).
Note that the mosaic types other than $U_1$ and $V_1$ arise
when the leading mosaic tile is near the boundary of the mosaic system.

\begin{figure}[h]
\includegraphics{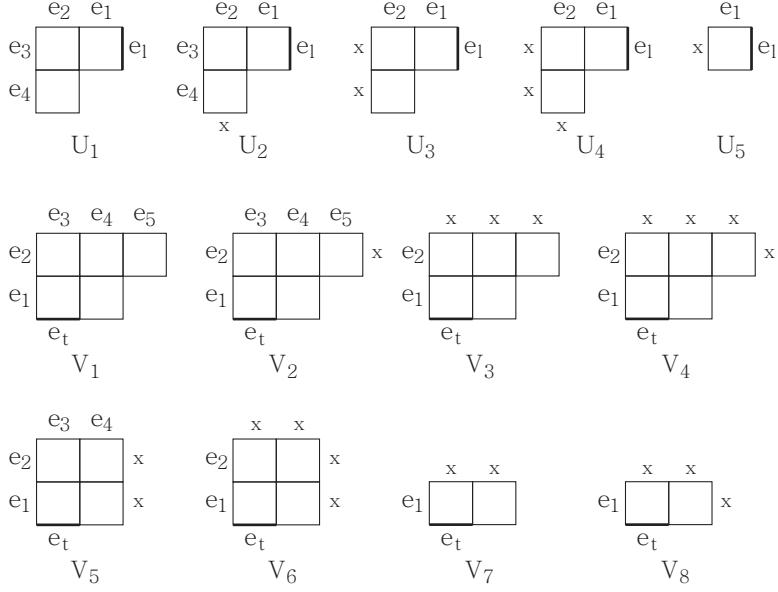}
\caption{Five types of $l$-cling mosaics and eight types of $t$-cling mosaics.}
\label{fig5}
\end{figure}

Now we define cp-ratios for each type of cling mosaics as follows.
We say that the associated contact edges $e_i$'s are {\em given\/}
if the presence of connection points of them are given.
For a type $U_k$ and given $e_i$'s, we define
$$ \mbox{{\em cp-ratio\/} of $U_k$} =
\frac{|\{ \mbox{type } U_k \mbox{ cling mosaics with the given } e_i \mbox{'s and } e_l = \mbox{o} \}|}
{|\{ \mbox{type } U_k \mbox{ cling mosaics with the given } e_i \mbox{'s and any } e_l \}|}.$$
And $u_k$ denotes the pair of the minimum and the maximum among all cp-ratios for the type $U_k$ 
that occur in any given $e_i$'s.
Similarly define the pair $v_{k'}$ for the type $V_{k'}$.

\begin{lemma} \label{lem:cp-ratio}
The pairs of cp-ratios for the thirteen types of cling mosaics are as follows:
$u_1 = \{ \frac{1}{4}, \frac{1}{2} \}$,
$u_2 = u_3 = u_4 = v_5 = v_6 = \{ \frac{1}{3}, \frac{1}{2} \}$,
$v_1 = \{ \frac{1}{4}, \frac{3}{5} \}$,
$v_2 = \{ \frac{1}{4}, \frac{4}{7} \}$,
$v_3 = v_4 = \{ \frac{4}{11}, \frac{1}{2} \}$, and
$u_5 = v_7 = v_8 = \{ \frac{1}{2}, \frac{1}{2} \}.$
\end{lemma}

\begin{proof}
First consider a submosaic $W$ consisting of three mosaic tiles $M_1$, $M_2$, and $M_3$
as drawn in the center of Figure~\ref{fig6}.
Each of $e_1 e_2$ and $e_3 e_4$ has four choices of the presence of connection points
among xx, xo, ox and oo.
Define $4 \times 4$ matrices $N_{c_1 c_2} = (n_{ij})$,
where $n_{ij}$ is the number of all possible suitably connected submosaics $W$
with the given $c_1 c_2$, the $i$-th $e_1 e_2$ and the $j$-th $e_3 e_4$
in the order of xx, xo, ox, and oo.
Then
{\footnotesize
$$N_{\mbox{xx}} = \begin{bmatrix}
2 & 2 & 2 & 2 \\
2 & 2 & 1 & 1 \\
2 & 2 & 2 & 2 \\
2 & 2 & 1 & 1
\end{bmatrix}, \
N_{\mbox{xo}} = \begin{bmatrix}
2 & 1 & 2 & 1 \\
2 & 1 & 1 & 1 \\
2 & 1 & 2 & 1 \\
2 & 1 & 1 & 1
\end{bmatrix},$$
$$N_{\mbox{ox}} = \begin{bmatrix}
2 & 2 & 2 & 2 \\
2 & 2 & 1 & 1 \\
1 & 1 & 1 & 1 \\
1 & 1 & 1 & 1
\end{bmatrix} \ \mbox{and} \ \
N_{\mbox{oo}} = \begin{bmatrix}
2 & 1 & 2 & 1 \\
2 & 1 & 1 & 1 \\
1 & 1 & 1 & 0 \\
1 & 0 & 1 & 1
\end{bmatrix}.$$
}

These four matrices can be obtained from the following two rules.
The first is that if $e_2 e_3$ is oo, then $M_3$ is $lt$-cp,
so it is uniquely determined by Lemma~\ref{lem:choice} and it must be $\hat{r}\hat{b}$-cp.
And if $e_2 e_3$ is not oo, then $M_3$ is $\widehat{lt}$-cp,
so it has two choices of mosaic tiles for given $e_2 e_3$,
one of which is $\hat{r}$-cp and the other is $r$-cp (similarly for $b$-cp).
The second rule is that, after $M_3$ is determined,
if $M_3$ is $\hat{r}$-cp, then $M_1$ is uniquely determined for given $c_1 e_1$.
And if $M_3$ is $r$-cp, then $M_1$ is uniquely determined when $c_1 e_1$ is not oo,
but there is no choice for $M_1$ when $c_1 e_1$ is oo.
The second rule can be applied to $M_2$ with $c_2 e_4$ in the same manner.

\begin{figure}[h]
\includegraphics{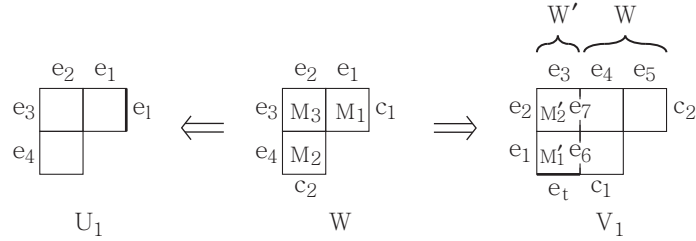}
\caption{Submosaic $W$ and modifying $W$ to $U_1$ and $V_1$.}
\label{fig6}
\end{figure}

For same sized matrices $A$ and $B$,
$\{ \frac{A}{B} \}$ denotes the pair consisting of the minimum and the maximum among all entries of the matrix obtained from dividing $A$ by $B$ entry-wise.
From now on, the mark~$*$ is used when we consider both x and o.
For examples, $N_{\mbox{o}*} = N_{\mbox{ox}} + N_{\mbox{oo}}$.

For the types $U_1$ through $U_4$, we use $W$ after identifying $c_1 = e_l$.
Each entry of $N_{\mbox{o}*}$ indicates the number of all possible
type $U_1$ cling mosaics with given $e_i$'s and $e_l=$ o,
and $N_{**}$ the number of type $U_1$ cling mosaics with given $e_i$'s and any $e_l$.
Note that there is no restriction on $c_2$.
Thus each entry of the matrix obtained from dividing $N_{\mbox{o}*}$ by $N_{**}$ entry-wise
is the cp-ratio for given $e_i$'s.
Now $u_1$ is the pair of the minimum and the maximum among all entries of this matrix.
Thus $u_1 = \{ \frac{N_{\mbox{\footnotesize{o}}*}}{N_{**}} \} =
\{ \frac{1}{4}, \frac{1}{2} \}$.
$u_2$ can be obtained by merely changing $N_{\mbox{o}*}$ and $N_{**}$
by $N_{\mbox{ox}}$ and $N_{*\mbox{x}}$, respectively,
because $c_2 =$ x.
Thus
$u_2 = \{ \frac{N_{\mbox{\footnotesize{ox}}}}{N_{*\mbox{\footnotesize{x}}}} \} = \{ \frac{1}{3}, \frac{1}{2} \}$.

The restriction $e_3 e_4 =$ xx for the types $U_3$ and $U_4$ is related to
only the first columns of the associated matrices.
The rest of the proof is similar to the previous case.
Thus,
$$ u_3 = \{ \frac{1\mbox{\footnotesize{st column of} } N_{\mbox{\footnotesize{o}}*}}
{1\mbox{\footnotesize{st column of} } N_{**}} \} = \{ \frac{1}{3}, \frac{1}{2} \}  \text{ and }
u_4 = \{ \frac{1\mbox{\footnotesize{st column of} } N_{\mbox{\footnotesize{ox}}}}
{1\mbox{\footnotesize{st column of} } N_{*\mbox{\footnotesize{x}}}} \} = \{ \frac{1}{3}, \frac{1}{2} \}. $$

For the types $V_1$ through $V_4$, we use $W$ again
after identifying $e_1$, $e_2$, $e_3$, and $e_4$ of $W$ with $e_6$, $e_7$, $e_4$, and $e_5$ of $V_i$'s,
respectively, combined with another submosaic $W'$ as shown in Figure~\ref{fig6}.
Define two $4 \times 8$ matrices $N^{(1)}_{e_t} = (n_{ij})$, for $e_t = $ x or o,
where $n_{ij}$ is the number of all possible submosaics $V_1$ with the given $e_t$, the $i$-th  $e_1 e_2$
and the $j$-th $e_3 e_4 e_5$ in the reverse dictionary order as before.
In the following matrices,
``$x$-th row" and ``$x \! + \! y$-th rows" mean
the $x$-th row of the previously obtained matrix $N_{**}$ and
the sum of the $x$-th row and the $y$-th row of $N_{**}$, respectively.
Then
{\footnotesize
$$N^{(1)}_{\mbox{x}} = \begin{bmatrix}
1 \! + \! 4\mbox{th rows} & 2 \! + \! 3\mbox{rd rows} \\
2 \! + \! 3\mbox{rd rows} & 1\mbox{st row} \\
2 \! + \! 3\mbox{rd rows} & 1 \! + \! 4\mbox{th rows} \\
1 \! + \! 4\mbox{th rows} & 3\mbox{rd row}
\end{bmatrix}
= \begin{bmatrix}
14 & 10 & 12 & 10 & 14 & 11 & 10 &  8 \\
14 & 11 & 10 &  8 &  8 &  6 &  8 &  6 \\
14 & 11 & 10 &  8 & 14 & 10 & 12 & 10 \\
14 & 10 & 12 & 10 &  6 &  5 &  6 &  4
\end{bmatrix}, $$
$$N^{(1)}_{\mbox{o}} = \begin{bmatrix}
2+3\mbox{rd rows} & 1+4\mbox{th rows} \\
1+4\mbox{th rows} & 3\mbox{rd row} \\
1\mbox{st row} & 2\mbox{nd row} \\
2\mbox{nd row} & 1\mbox{st row}
\end{bmatrix}
= \begin{bmatrix}
14 & 11 & 10 &  8 & 14 & 10 & 12 & 10 \\
14 & 10 & 12 & 10 &  6 &  5 &  6 &  4 \\
 8 &  6 &  8 &  6 &  8 &  6 &  4 &  4 \\
 8 &  6 &  4 &  4 &  8 &  6 &  8 &  6
\end{bmatrix}. $$
}

For example, we will compute the second row of $N^{(1)}_{\mbox{x}}$,
and the reader can find the remaining rows in the same manner.
For this case, $e_t=$ x, $e_1 e_2=$ xo,
the left four entries of this row are related to $e_3=$ x,
and the right four entries are related to $e_3=$ o.
If $e_3=$ x, then the pair $M'_1$ and $M'_2$ of $W'$
has two choices, such as $M'_1 = T_1$ and $M'_2 = T_6$, or $M'_1 = T_4$ and $M'_2 = T_2$.
Therefore $e_6 e_7$ must be xo or ox, respectively.
These two cases are related to the second and the third rows of $N_{**}$, respectively.
Thus the numbers of all possible such $W$ for each $e_4 e_5$ are represented by the sum of these two rows.
If $e_3=$ o, then this pair has unique choice of $M'_1 = T_1$ and $M'_2 = T_5$,
and so $e_6 e_7$ must be xx.
It is related to the first row of $N_{**}$, which represents the numbers of all such $W$ for each $e_4 e_5$.
Each entry of $N^{(1)}_{\mbox{o}}$ indicates the number of all possible
type $V_1$ $t$-cling mosaics with given $e_i$'s and $e_t=$ o,
and $N^{(1)}_{*}$ the number of type $V_1$ $t$-cling mosaics with given $e_i$'s and any $e_t$.
Now we get the cp-ratio for given $e_i$'s in the same way as previous.
Thus, 
$$ v_1 = \{ \frac{N^{(1)}_{\mbox{\footnotesize{o}}}}{N^{(1)}_{*}} \} = \{ \frac{1}{4}, \frac{3}{5} \}. $$

For $V_2$, define other two $4 \times 8$ matrices $N^{(2)}_{e_t}$, for $e_t = $ x or o.
$N^{(2)}_{\mbox{x}}$ and $N^{(2)}_{\mbox{o}}$ are obtained in the same manner as
computing $N^{(1)}_{\mbox{x}}$ and $N^{(1)}_{\mbox{o}}$ after replacing $N_{**}$ by $N_{*\mbox{x}}$,
since $c_2=$ x.
Then
{\footnotesize
$$N^{(2)}_{\mbox{x}} = \begin{bmatrix}
7 & 7 & 6 & 6 & 7 & 7 & 5 & 5 \\
7 & 7 & 5 & 5 & 4 & 4 & 4 & 4 \\
7 & 7 & 5 & 5 & 7 & 7 & 6 & 6 \\
7 & 7 & 6 & 6 & 3 & 3 & 3 & 3
\end{bmatrix} \mbox{ and }
N^{(2)}_{\mbox{o}} = \begin{bmatrix}
7 & 7 & 5 & 5 & 7 & 7 & 6 & 6 \\
7 & 7 & 6 & 6 & 3 & 3 & 3 & 3 \\
4 & 4 & 4 & 4 & 4 & 4 & 2 & 2 \\
4 & 4 & 2 & 2 & 4 & 4 & 4 & 4
\end{bmatrix}.$$
}

Then $v_2$ can be obtained from merely changing $N^{(1)}_{\mbox{o}}$ and $N^{(1)}_{*}$
by $N^{(2)}_{\mbox{o}}$ and $N^{(2)}_{*}$, respectively.
Thus, 
$$ v_2 = \{ \frac{N^{(2)}_{\mbox{\footnotesize{o}}}}{N^{(2)}_{*}} \} = \{ \frac{1}{4}, \frac{4}{7} \}. $$

The restriction $e_3 e_4 e_5=$ xxx for the types $V_3$ and $V_4$ is related to
only the first columns of the associated matrices.
Thus, 
$$ v_3 = \{ \frac{1\mbox{\footnotesize{st column of} } N^{(1)}_{\mbox{\footnotesize{o}}}}
{1\mbox{\footnotesize{st column of} } N^{(1)}_{*}} \} = \{ \frac{4}{11}, \frac{1}{2} \} \text{ and }
v_4 = \{ \frac{1\mbox{\footnotesize{st column of} } N^{(2)}_{\mbox{\footnotesize{o}}}}
{1\mbox{\footnotesize{st column of} } N^{(2)}_{*}} \} = \{ \frac{4}{11}, \frac{1}{2} \}. $$

Consider the types $V_5$ and $V_6$.
Define two $4 \times 4$ matrices $N^{(3)}_{e_t} = (n_{ij})$, for $e_t = $ x or o,
where $n_{ij}$ is the number of all possible submosaics $V_5$ with the given $e_t$, the $i$-th  $e_1 e_2$
and the $j$-th $e_3 e_4$.
Using the same manner of computing the associated matrices at the beginning of the proof,
the reader can find the matrices $N^{(3)}_{\mbox{x}}$ and $N^{(3)}_{\mbox{o}}$ as follows:
{\footnotesize
$$N^{(3)}_{\mbox{x}} = \begin{bmatrix}
2 & 2 & 2 & 2 \\
2 & 2 & 1 & 1 \\
2 & 2 & 2 & 2 \\
2 & 2 & 1 & 1
\end{bmatrix} \mbox{ and }  \
N^{(3)}_{\mbox{o}} = \begin{bmatrix}
2 & 2 & 2 & 2 \\
2 & 2 & 1 & 1 \\
1 & 1 & 1 & 1 \\
1 & 1 & 1 & 1
\end{bmatrix}.$$
}
From the same calculation as before,
$$ v_5 = \{ \frac{N^{(3)}_{\mbox{\footnotesize{o}}}}{N^{(3)}_{*}} \} =
\{ \frac{1}{3}, \frac{1}{2} \} \text{ and }
v_6 = \{ \frac{1\mbox{\footnotesize{st column of} } N^{(3)}_{\mbox{\footnotesize{o}}}}
{1\mbox{\footnotesize{st column of} } N^{(3)}_{*}} \} = \{ \frac{1}{3}, \frac{1}{2} \}. $$

For the remaining types, $u_5$, $v_7$, and $v_8$ are obtained by counting directly
for each case of $e_1 =$  x or o, as $u_5 = v_7 = v_8 = \{ \frac{1}{2}, \frac{1}{2} \}$.
\end{proof}

\section{Proof of Theorem~\ref{thm:polygon}}

We will compute lower and upper bounds of the growth ratio at each leading mosaic tile
by using the cp-ratios of the associated cling mosaics.
Let $M_{i,j}$ be a leading mosaic tile with the associated $l$- and $t$-cling mosaics $U_k$ and $V_{k'}$.
Let $S_{kk'}$ and $L_{kk'}$ denote the multiplication of the smallest (resp. largest)
elements of $u_k$ and $v_{k'}$.

\begin{lemma} \label{lem:ratio}
For $i \neq 1,m$ and $j \neq 1,n$,
$2 - L_{kk'} \leq r_{ij} \leq 2 - S_{kk'}$.
\end{lemma}

\begin{proof}
Suppose that $i \neq 1,m$ and $j \neq 1,n$.
Recall that an $(i,j)$-quasimosaic in $Q_{i,j}$ is obtained from a $a(i,j)$-quasimosaic in $Q_{a(i,j)}$
by attaching a proper leading mosaic tile $M_{i,j}$.
This mosaic tile should be suitably connected
according to the presence of connection points on its left and top edges.
In this stage, there are two possibilities, as follows:
if $M_{i,j}$ is $\widehat{lt}$-cp, then it has two choices,
and if it is $lt$-cp, then it has a unique choice.
Therefore, for given cling mosaics, $M_{i,j}$ has a unique choice only when $e_l e_t =$ oo.

Consider a submosaic consisting of $M_{i,j}$ and $l$- and $t$-cling mosaics.
Assume that the presence of connection points on all contact edges $e_i$'s are given.
Then
$$ \frac{|\{(i,j) \mbox{-quasimosaics with the given } e_i \mbox{'s} \}|}
{|\{a(i,j) \mbox{-quasimosaics with the given } e_i \mbox{'s} \}|} = $$
$$ \frac{|\{\mbox{submosaics consisting of } M_{i,j}
\mbox{ and the a.c.m.'s with the given } e_i \mbox{'s} \}|}
{|\{\mbox{submosaics consisting of only the a.c.m.'s with the given } e_i \mbox{'s} \}|}, $$
where a.c.m. means associated cling mosaic.

Let $c_k$ and $c'_{k'}$ denote the associated cp-ratios of the $l$- and $t$-cling mosaics
for the given contact edges $e_i$'s.
Then the latter quotient of the equality is $2 \times (1 - c_k c'_{k'}) + 1 \times (c_k c'_{k'})
= 2 - c_k c'_{k'}$.
Furthermore, $2 - c_k c'_{k'}$ must lie between $2 - L_{kk'}$ and $2 - S_{kk'}$,
so is the former quotient.
Therefore, $r_{ij}$ lies between $2 - L_{kk'}$ and $2 - S_{kk'}$.
\end{proof}

\begin{lemma} \label{lem:polygon}
Let $m$ and $n$ be integers with $3 \leq m \leq n$.

For $m=3$,
$14 \left(\frac{7}{2}\right)^{n-3} -1 \leq p_{3 \times n} \leq 14 \left(\frac{11}{3} \right)^{n-3} -1$.

For $m=4$,
$8 \left(\frac{49}{8}\right)^{n-2} -1 \leq p_{4 \times n} \leq \frac{9520}{27} \left(\frac{155}{22} \right)^{n-4} -1$.

For $m \geq 5$,
$8 \cdot 6^{m-4} \left(\frac{49}{8}\right)^{n-2} \left(\frac{17}{10}\right)^{(m-4)(n-4)} -1 \leq p_{m \times n}$

\hspace{10mm} and \ $p_{m \times n} \leq \frac{337280}{1863} \left(\frac{2645}{192} \right)^{m-4}
\left(\frac{2415}{176} \right)^{n-4} \left(\frac{31}{16} \right)^{(m-5)(n-5)} -1$.
\end{lemma}

\begin{proof} 
First we handle the general case that $5 \leq m < n$.
Consider a leading mosaic tile $M_{i,j}$ for $4 \leq i \leq m-2$ and $4 \leq j \leq n-3$.
Associated $l$- and $t$-cling mosaics are of types $U_1$ and $V_1$, respectively,
because they are apart from the boundary of the mosaic system.
Since the smallest cp-ratios in $u_1$ and $v_1$ are both $\frac{1}{4}$
and their largest cp-ratios are $\frac{1}{2}$ and $\frac{3}{5}$, respectively,
$r_{ij}$ lies between $2 - L_{11} = \frac{17}{10}$ and $2 - S_{11} = \frac{31}{16}$.
For the remaining leading mosaic tiles, one or both of their associated cling mosaics
are attached to the boundary of the mosaic system.

A chart in Figure~\ref{fig7}, called the {\em cling mosaic chart\/},
illustrates all possible combinations of cling mosaics at each position of leading mosaic tile.
For example, at the position of the leading mosaic tile $M_{3,2}$,
the associated $l$- and $t$-cling mosaics are of types $U_5$ and $V_3$, respectively.

\begin{figure}[h]
\includegraphics{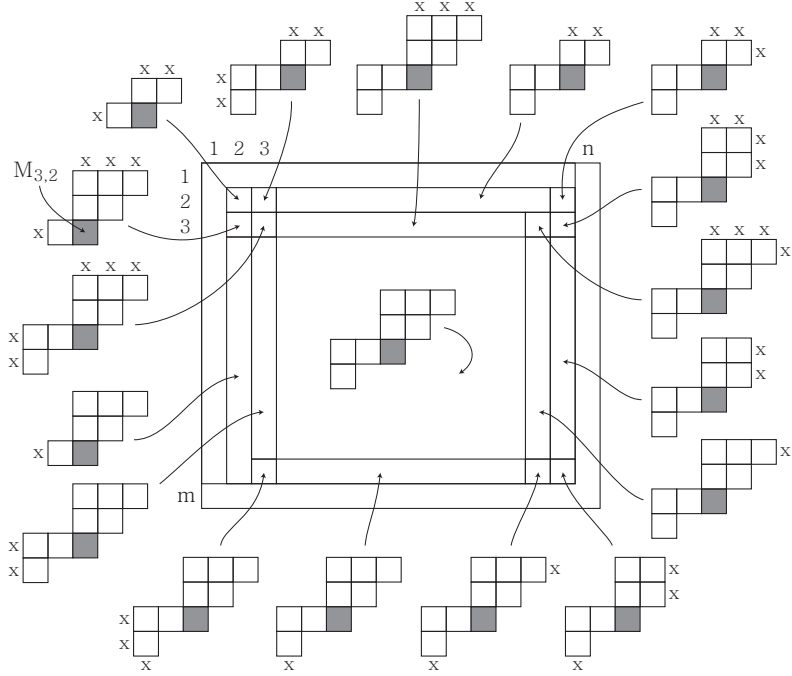}
\caption{Cling mosaic chart for the general case.}
\label{fig7}
\end{figure}

From Lemmas~\ref{lem:cp-ratio} and \ref{lem:ratio} combined with the cling mosaic chart,
we get Table~\ref{tab1}, called the {\em growth ratio table\/}.
Each row explains the placements of leading mosaic tiles $M_{i,j}$, 
the associated multiplications $u_k \cdot v_{k'}$ of cp-ratios,
possible variance of the related growth ratios $r_{i,j}$, and the number of the related mosaic tiles.

Note that for $i=1$ ($j \neq n$), the leading mosaic tile $M_{1,j}$ must be $\hat{t}$-cp.
Assume that $M_{1,j-1}$ is already decided.
Then $M_{1,j}$ has exactly two choices by Lemma~\ref{lem:choice}, so $r_{1j} = 2$.
Similarly,we get $r_{i1} = 2$ for $j=1$ ($i \neq m$).
And for $i=m$, $M_{m,j}$ must be $\hat{b}$-cp.
Assume that $M_{m,j-1}$ and $M_{m-1,j}$ are already decided.
But in any case, $M_{m,j}$ is determined uniquely, so $r_{mj} = 1$.
Similarly we get $r_{in} = 1$ for $j=n$.
Indeed, the method in this paragraph works for all the cases of $3 \leq m \leq n$.
\vspace{3mm}

\renewcommand{\arraystretch}{1.2}

\begin{table}[h]
{\footnotesize
\begin{tabular}{cccc}      \hline \hline
$(i,j)$ of $M_{i,j}$ & \ \ $u_k \cdot v_{k'}$ \ \ & $r_{i,j}$ & \ {\em number of tiles\/} \ \\ \hline
$i=1$ or $j=1$ except $(1,n),(m,1)$ &   & 2 & $m+n-3$ \\
$i=m$ or $j=n$ &   & 1 & $m+n-1$ \\
$4 \leq i \leq m-2$ and $4 \leq j \leq n-3$ & $u_1 \cdot v_1$ &
$\frac{17}{10} \sim \frac{31}{16}$  & $(m-5)(n-6)$ \\
$(2,2)$ & $u_5 \cdot v_7$ & $\frac{7}{4}$ & 1 \\
$(2,3)$ & $u_3 \cdot v_7$ & $\frac{7}{4} \sim \frac{11}{6}$ & 1 \\
$i=2$ and $4 \leq j \leq n-2$ & $u_1 \cdot v_7$ & $\frac{7}{4} \sim \frac{15}{8}$ & $n-5$ \\
$(2,n-1)$ & $u_1 \cdot v_8$ & $\frac{7}{4} \sim \frac{15}{8}$  & 1 \\
$(3,2)$ & $u_5 \cdot v_3$ & $\frac{7}{4} \sim \frac{20}{11}$ & 1 \\
$(3,3)$ & $u_3 \cdot v_3$ & $\frac{7}{4} \sim \frac{62}{33}$ & 1 \\
$i=3$ and $4 \leq j \leq n-3$ & $u_1 \cdot v_3$ & $\frac{7}{4} \sim \frac{21}{11}$ & $n-6$ \\
$(3,n-2)$ & $u_1 \cdot v_4$ & $\frac{7}{4} \sim \frac{21}{11}$ & 1 \\
$(3,n-1)$ & $u_1 \cdot v_6$ & $\frac{7}{4} \sim \frac{23}{12}$ & 1 \\
$4 \leq i \leq m-1$ and $j=2$ & $u_5 \cdot v_1$ & $\frac{17}{10} \sim \frac{15}{8}$ & $m-4$ \\
$4 \leq i \leq m-2$ and $j=3$ & $u_3 \cdot v_1$ & $\frac{17}{10} \sim \frac{23}{12}$ & $m-5$ \\
$4 \leq i \leq m-2$ and $j=n-2$ & $u_1 \cdot v_2$ & $\frac{12}{7} \sim \frac{31}{16}$ & $m-5$ \\
$4 \leq i \leq m-2$ and $j=n-1$ & $u_1 \cdot v_5$ & $\frac{7}{4} \sim \frac{23}{12}$ & $m-5$ \\
$(m-1,3)$ & $u_4 \cdot v_1$ & $\frac{17}{10} \sim \frac{23}{12}$ & 1 \\
$i=m-1$ and $4 \leq j \leq n-3$ & $u_2 \cdot v_1$ & $\frac{17}{10} \sim \frac{23}{12}$ & $n-6$ \\
$(m-1,n-2)$ & $u_2 \cdot v_2$ & $\frac{12}{7} \sim \frac{23}{12}$ & 1 \\
$(m-1,n-1)$ & $u_2 \cdot v_5$ & $\frac{7}{4} \sim \frac{17}{9}$ & 1 \\ \hline \hline
\end{tabular}
}
\vspace{4mm}
\caption{Growth ratio table for the general case.}
\label{tab1}
\end{table}

The chart in Figure~\ref{fig8} illustrates bounds of the growth ratios at each position of leading mosaic tile
according to the growth ratio table.
This is called the {\em growth ratio chart\/}.

\begin{figure}[h]
\includegraphics{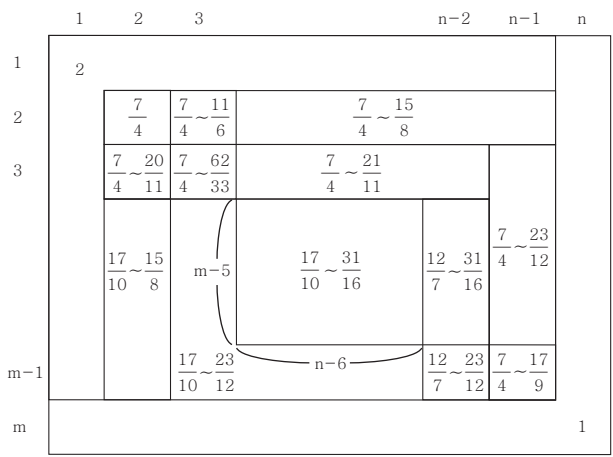}
\caption{Growth ratio chart for the general case.}
\label{fig8}
\end{figure}

From the growth ratio chart for $5 \leq m < n$, we get rigorous lower and upper bounds of $p_{m \times n}$,
which are obtained by merely multiplying every growth ratio at each leading mosaic tile
and subtracting by 1 as in equation~(\ref{eq:1}).
Thus, we have 
$$ 8 \cdot 6^{m-4} \left(\frac{49}{8}\right)^{n-2} \left(\frac{17}{10}\right)^{(m-4)(n-4)} -1 \leq p_{m \times n}, $$
$$ p_{m \times n} \leq \frac{337280}{1863} \left(\frac{2645}{192} \right)^{m-4}
\left(\frac{2415}{176} \right)^{n-4} \left(\frac{31}{16} \right)^{(m-5)(n-5)} -1. $$

For the remaining cases $m = 3$, $m=4$, and $m=n=5$, 
the reader may draw the associated cling mosaic charts and
compute the growth ratio tables.
Then the related growth ratio charts will be obtained as shown in Figure~\ref{fig9}.
Furthermore,
\begin{equation*} 
\begin{split}
14 \left(\frac{7}{2}\right)^{n-3} -1 \leq p_{3 \times n} \leq 14 \left(\frac{11}{3} \right)^{n-3} -1 
\hspace{9mm} & \text{ for } m=3, \text{ and } \\
8 \left(\frac{49}{8}\right)^{n-2} -1 \leq p_{4 \times n} \leq \frac{9520}{27} \left(\frac{155}{22} \right)^{n-4} -1 
\hspace{2mm} & \text{ for } m=4.
\end{split}
\end{equation*}
Indeed for the case of $m=n=5$, we eventually get the same result 
as in the general case, by applying $m=n=5$.
\end{proof}

\begin{figure}[h]
\includegraphics{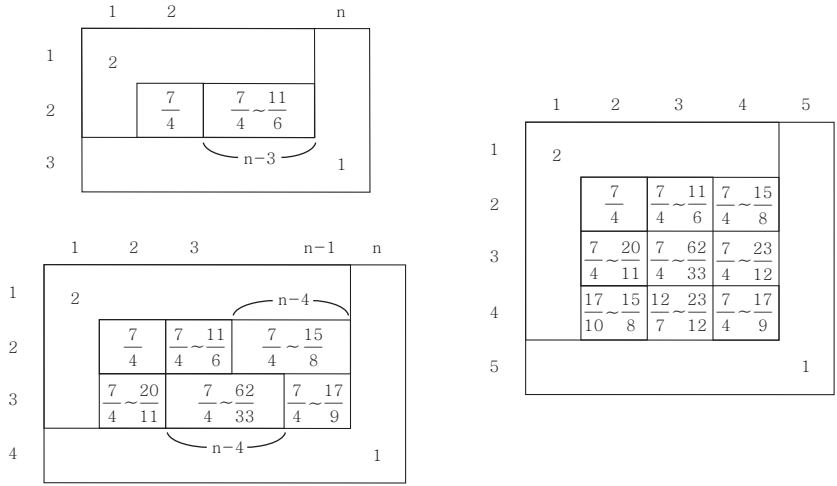}
\caption{Three growth ratio charts for $m = 3$, $m=4$, and $m=n=5$ from the top left to the right.}
\label{fig9}
\end{figure}

\begin{proof}[Proof of Theorem~\ref{thm:polygon}]
The result follows directly from Lemma~\ref{lem:polygon} after loosing the bounds slightly.
Speaking precisely, for any case of $3 \leq m \leq n$,
if $i \neq 1,m$ and $j \neq 1,n$, 
then $r_{ij}$ always lies between $\frac{17}{10}$ and $\frac{31}{16}$.
Furthermore,
if $i=1$ or $j=1$, except $(1,n)$ and $(m,1)$, then $r_{ij} = 2$, and
if $i=m$ or $j=n$, then $r_{ij} = 1$.
Therefore,
$$2^{m+n-3} \left(\frac{17}{10}\right)^{(m-2)(n-2)} -1 \ \leq \ p_{m \times n} \ \leq \
2^{m+n-3} \left(\frac{31}{16}\right)^{(m-2)(n-2)} -1.$$
Note that $-1$ can be ignored for the brief formula,
since this inequality is obtained from Lemma~\ref{lem:polygon} after loosening the bounds slightly.
\end{proof}

\end{document}